\documentclass{article}

\usepackage{amsfonts}
\usepackage{amsmath,ntheorem}
\usepackage{amssymb}
\usepackage{graphics}
\usepackage{graphicx}
\usepackage{color}
\usepackage{float}
\usepackage{epsfig}
\usepackage{geometry}
\usepackage{indentfirst,latexsym,bm}
\usepackage{lscape}
\usepackage{tabularx}
\usepackage{enumerate}
\usepackage{longtable}
\usepackage{lscape}
\usepackage{mathrsfs}
\usepackage{bbm}

\numberwithin{equation}{section}

\newtheorem{theorem}{Theorem}[section]

\newtheorem*{assumption}{Assumption}

\newtheorem{lemma}[theorem]{Lemma}

\newtheorem{proposition}[theorem]{Proposition}

\newenvironment{proof}[1][Proof]{\noindent\textbf{#1.}\:}{\hfill $\square$}

\def\EE{\mathbb{E}}
\def\EQ{\mathbb{E}^{\mathbb{Q}}}
\def\Et{\mathbb{E}_{\mathcal{F}_t}}

\def\rm{\mathbb{R}^m}

\title{Quantitative stability and numerical analysis of Markovian quadratic BSDEs with reflection\thanks{
The authors thank the editor and both referees for their careful
reading and helpful comments.}}

\author{
Dingqian Sun\thanks{Department of Finance and Control Sciences,
School of Mathematical Sciences, Fudan University, Shanghai 200433,
China. Partially supported by China Scholarship Council.
\texttt{Email:dqsun14@fudan.edu.cn} } \and Gechun
Liang\thanks{Department of Statistics, University of Warwick,
Coventry, CV4 7AL, U.K. Partially supported by the National Natural
Science Foundation of China (No. 12171169) and Guangdong Basic and
Applied Basic Research Foundation (No. 2019A1515011338). GL thanks
J. F. Chassagneux and A. Richou for helpful and inspiring
discussions on how to extend to the state dependent volatility case.
\texttt{Email:g.liang@warwick.ac.uk}} \and Shanjian
Tang\thanks{Department of Finance and Control Sciences, School of
Mathematical Sciences, Fudan University, Shanghai 200433, China.
Partially supported by National Science Foundation of China (No.
11631004) and National Key R\&D Program of China (No.
2018YFA0703903). \texttt{Email:sjtang@fudan.edu.cn}}}

\date{}

\begin{document}
\maketitle

\begin{abstract}
We study the quantitative stability of the solutions to Markovian quadratic reflected BSDEs with bounded terminal data. By virtue of BMO
martingale and change of measure techniques, we obtain stability estimates for the variation of the solutions with different underlying forward processes. In addition, we propose a truncated discrete-time numerical scheme for quadratic reflected BSDEs, and obtain the
explicit rate of convergence by applying the quantitative stability result.\\

\noindent\textit{Keywords}: Quadratic BSDE with reflection, \and stability of solutions, \and discretely reflected BSDE, \and rate of
convergence\\
\end{abstract}

\section{Introduction}
In this paper, we are interested in the stability of the solutions to the following quadratic reflected backward stochastic differential
equations (BSDEs for short) under Markovian framework
\begin{align}\label{Y}
\begin{split}
&Y_t=g(X_T)+\int^T_t f(s,X_s,Y_s,Z_s)ds-\int^T_t Z_s dW_s +K_T-K_t, \\
&Y_t\geqslant g(X_t),\quad \int^T_0 (Y_t-g(X_t))dK_t=0,
\end{split}
\end{align}
where $T>0$ is a fixed finite time horizon and the underlying forward process solves
\begin{equation}\label{X}
X_t=x+\int^t_0 b(s,X_s)ds+\int^t_0 \sigma(s)dW_s,\quad t\in[0,T].
\end{equation}
Herein, $\{W_t\}_{t\geqslant 0}$ is an $m$-dimensional standard Brownian motion defined on a complete probability space
$(\Omega,\mathcal{F},\mathbb{P})$ and $\{\mathcal{F}_t\}_{t\geqslant 0}$ is the augmented natural filtration of $W$ which satisfies the usual
conditions. Let $\mathcal{P}$ denote the progressively measurable $\sigma$-field on $[0,T]\times\Omega$.

Throughout this paper, suppose that all the coefficients $b,\sigma, g$ and $f$ are deterministic and continuous functions and $b:
[0,T]\times\mathbb{R}^n\rightarrow\mathbb{R}^n$, $\sigma: [0,T]\rightarrow\mathbb{R}^{n\times m}$ satisfy, for all $ t\in[0,T]$ and $x,x^\prime \in \mathbb{R}^n$,
that,
\begin{equation}\tag{HX}
\begin{split}
\vert b(t,0)\vert+\vert \sigma(t)\vert & \leqslant L, \\
\vert b(t,x)-b(t,x^{\prime})\vert & \leqslant L\vert x-x^{\prime}\vert,
\end{split}
\end{equation}
for a positive constant $L$. We also assume that $g: \mathbb{R}^n\rightarrow\mathbb{R}$ satisfies Lipschitz condition $\vert
g(x)-g(x^{\prime})\vert \leqslant L\vert x-x^{\prime}\vert$ for all $x,x^\prime \in \mathbb{R}^n$ and is bounded by $M_g$, while $f:
[0,T]\times\mathbb{R}^n\times\mathbb{R}\times\mathbb{R}^{n\times m}\rightarrow\mathbb{R}$ is Lipschitz with respect to $y$ and locally Lipschitz with respect to
both $x$ and $z$, and has at most quadratic growth with respect to $z$, i.e., for any $t\in[0,T]$ and $ (x,y,z),(x^\prime,y^\prime,z^\prime) \in
\mathbb{R}^n\times \mathbb{R}\times\mathbb{R}^{n\times m}$,
\begin{equation}\tag{HF}
\begin{split}
\vert f(t,x,y,z)\vert &\leqslant M_f(1+\vert y\vert)+\frac{\alpha}{2}\vert z\vert^2,\\
\vert f(t,x,y,z)-f(t,x^{\prime},y,z)\vert & \leqslant L(1+\vert z\vert)\vert x-x^{\prime}\vert,  \\
\vert f(t,x,y,z)-f(t,x,y^{\prime},z)\vert & \leqslant L\vert y-y^{\prime}\vert,  \\
\vert f(t,x,y,z)-f(t,x,y,z^{\prime})\vert & \leqslant L(1+\vert z\vert+\vert z^{\prime}\vert)\vert z-z^{\prime}\vert,
\end{split}
\end{equation}
where $M_g$, $M_f$ and $\alpha$ are all positive constants.

Thanks to the seminal work \cite{KobyBSDE}, \cite{PB1} and \cite{PB}, the existence and uniqueness of the solution for the corresponding quadratic BSDEs (without reflection) has been well developed. The reflected case was studied in \cite{Koby} with bounded terminal value and obstacle, and \cite{ErhanSong}, \cite{LX} for
unbounded cases. In addition to the existence and uniqueness of the solution, the stability is also an important property that focuses on the variation of the solutions under small perturbations of the coefficients. It is widely used to obtain continuity properties of the solutions. In this paper, we will apply it to the numerical analysis of quadratic reflected BSDEs.

Under the Lipschitz setting, a basic stability result has been developed in \cite[Proposition 3.6]{El}, which gives the variation of the
solutions in terms of the suitable norms of their terminal values, generators and obstacles. Based on this result, \cite{Ma} study the
$L^2$-modulus regularity of the martingale integrand $Z$ via a Feynman–Kac type formula and give both the numerical scheme in the spirit of
Bermuda options and its rate of convergence. \cite{BBJC} further apply the stability result to approximate $(Y, Z)$ by its counterpart
$(Y^e, Z^e)$ constructed with the Euler scheme $X^\pi$ of (\ref{X}) and retrieve the convergence with the aid of a representation of the solution component $Z$
in terms of the next reflection time, removing the uniform ellipticity condition on $X$ in \cite{Ma}.


However, the counterpart of \cite[Proposition 3.6]{El} under
the quadratic setting is still lacking. The existing stability results focus on the continuity of the solutions. For example,
in \cite{KobyBSDE} (without reflection) and \cite{Koby} (with reflection), the authors show the uniform convergence of the solutions $(Y^n,Z^n)$ with
parameters $(g^n,f^n)$ to the solution $(Y,Z)$ with parameters $(g,f)$ when the obstacles $g^n$ and
generators $f^n$ uniformly converge to $g$ and $f$, by means of the
comparison theorem, monotone property and Lebesgue's theorem. However, the above continuity result does not say anything about the
quantitative dependence of the variation of the solutions on those parameters, which will play a pivotal role in the numerical analysis of quadratic BSDEs with reflection.

The main purpose of this paper is therefore to give, for the first time, \emph{a quantitative stability result on the solutions of the Markovian quadratic BSDE with
reflection (\ref{Y}) and apply this new stability result to establish the convergence of a truncated discrete-time numerical scheme for (\ref{Y})}. Proceeding under
Markovian framework, we shall mainly focus on the perturbations of the parameters in the forward process (\ref{X}) and study the variation of
the solutions $(Y,Z)$ to the quadratic reflected BSDE (\ref{Y}) driven by different forward processes.

Due to the quadratic growth condition, we will work with bounded terminal data in order to further take advantage of the properties of BMO
martingales, the latter of which is used ubiquitously in the numerical analysis of quadratic BSDEs without reflection, see \cite{JCAR},
\cite{IDR} and \cite{Richou2011} for example.
To be more specific, we first obtain some fundamental properties of the solution to the quadratic BSDE with reflection (\ref{Y}), that is, the
BMO property of the martingale integrand $Z*W$ and the $L^p$-integrability of $\int^T_0 |Z_t|^2 dt$ and $K_T$. Next, working under a new
equivalent probability measure induced by $Z*W$, we utilize the reverse H\"{o}lder inequality to obtain the estimate of the variation of the solution component $Y$
for any order in terms of the difference of underlying forward processes, followed by the estimates on the solution components $(Z,K)$ equipped with
appropriate norms. Finally, transferring back via John-Nirenberg inequality, we obtain the explicit dependence of the variation of the
solutions under the original probability measure. See Theorem 3.2 for further details.

Furthermore, we apply the stability result to the numerical analysis of quadratic reflected BSDEs. In contrast to quadratic BSDEs without reflection and Lipschitz BSDEs with reflection, where
the solution component $Z$ is typically bounded in Markovian setup, \emph{the solution component $Z$ for the quadratic reflected BSDE (1.1)-(1.2) is not necessarily bounded}. This is the major difficulty to
propose a numerical scheme and study its convergence. To overcome this difficulty, we reply on the discretely reflected BSDE (\ref{Y^R}) introduced in Section 4.
Thanks to the previous work \cite{DQ}, we can readily extend the results therein to obtain a uniform estimate of
$Z^\mathcal{R}$, the second component of the solution to the discretely reflected BSDE (\ref{Y^R}), and the convergence rate from the
discretely to continuously reflected BSDEs. In turn, we truncate the generator via the bound of $Z^\mathcal{R}$ and propose a truncated
discrete-time numerical scheme. This enables us to directly apply the existing numerical result under the Lipschitz setting (see \cite{BBJC}) to get the
approximation error for the discretely reflected BSDE with quadratic growth. However, when extending the estimates to the continuously
reflected case, a troublesome term $\kappa |\pi|$ appears and it will degenerate to a constant when to get the overall convergence rate (see Lemma \ref{Quaresult}). To overcome this difficulty, we introduce
$Z^{\mathcal{R},e}$ as defined in (\ref{Y^Re}), which is based on the Euler scheme for $X^\pi$, the same forward process as in our
discrete-time numerical scheme. The price to pay is that one needs to estimate an additional error between the solutions $(Y,Z)$ and $(Y^e,Z^e)$ of the
continuously reflected BSDEs driven by $X$ and $X^\pi$ respectively. It turns out this error can be controlled by applying the quantitative stability estimate. See Theorem 4.3 for further details. \\

The rest of this paper is organized as follows. In Section 2, we obtain some useful properties of the solution to the quadratic reflected BSDE
with bounded terminal value. The quantitative stability result under Markovian framework is derived in Section 3, with the assistance of powerful
techniques from BMO martingales. In the following section, we propose a truncated discrete-time numerical scheme for the quadratic reflected BSDE and
apply the stability result to obtain  a convergence rate for such a discrete-time approximation. Section 5 then concludes.

\section{Preliminaries}

In this section, we introduce the notations of different spaces and recall some known results on quadratic reflected BSDEs with bounded terminal data.

Without loss of generality, we assume that the forward process $X$ has dimension $n=1$ in the rest of the paper. Note that this is merely for the sake of notational simplicity. Let $\mathbb{S}^\infty[0,T]$ denote the set of $\mathbb{R}$-valued progressively measurable bounded processes, and
$\mathbb{K}^2[0,T]$ denote all the $\mathbb{R}$-valued continuous adapted processes $(K_t)_{0\leqslant t\leqslant T}$, which are increasing
with $K_0$=0 and $\EE\vert K_T\vert^2<\infty$. For $1\leqslant p<\infty$, $\mathbb{S}^p[0,T]$ denotes all $\mathbb{R}$-valued adapted
processes $(Y_t)_{0\leqslant t\leqslant T}$ such that $\Vert Y\Vert^p_{\mathbb{S}^p}:=\EE(\sup_{0\leqslant t\leqslant T} \vert Y_t\vert^p)<\infty$, and $\mathbb{H}^{p}([0,T];\rm)$ denotes all $\rm$-valued adapted processes $(Z_t)_{0\leqslant t\leqslant T}$ satisfying $\Vert Z\Vert^p_{\mathbb{H}^p}:=
\EE[\int^{T}_0 \vert Z_t
\vert^2_{\rm}dt]^{p/2}<\infty$. Moreover, $\mathbb{L}^p(\mathcal{F}_t)$ denotes all $\mathbb{R}$-valued $\mathcal{F}_t$-measurable variables satisfying $\Vert Y\Vert^p_{\mathbb{L}^p}:=\EE\vert Y_t\vert^p<\infty$ for any $t\in[0,T]$ and we usually omit $(\mathcal{F}_t)$ hereafter in case there is no ambiguity.

Under the above assumptions (HX) and (HF), we know the decoupled system (\ref{Y}) and (\ref{X}) with bounded terminal function and bounded obstacle has a
unique solution $(X,Y,Z,K)\in \mathbb{S}^2[0,T]\times\mathbb{S}^\infty[0,T]\times\mathbb{H}^2([0,T];\rm)\times\mathbb{K}^2[0,T]$, and
furthermore we denote $\Vert Y\Vert_\infty \triangleq M$. For more details of this result, we refer to \cite{Koby}. In the following, unless
otherwise specified, we shall use $C$ to denote the universal constant that may depend on all the given coefficients $L, T, M_g, M_f$ and
$\alpha$, and $C_p$ further depends on an extra parameter $p\geqslant 1$.

Next, we recall the definition and some basic properties of BMO martingales, which provide the techniques for our following study. For the
detailed theory, we refer the reader to \cite{Kaza}. We say a continuous local martingale $(M_t)_{t\in[0,T]}$ is a BMO martingale if it is
square-integrable with $M_0=0$ such that
$$\Vert M\Vert^2_{BMO}:=\sup_{\tau\in\mathcal{T}[0,T]}\Vert\EE[\langle M\rangle_T-\langle M\rangle_\tau|{\mathcal{F}_\tau}]\Vert_{\infty}< \infty,$$
where $\mathcal{T}[0,T]$ is the set of all stopping times valued in $[0,T]$.

\begin{lemma}\label{BMO}
Let $M$ be a BMO martingale. Then, we have:\\
$\mathrm{1)}$ The stochastic exponential
$$\mathcal{E}(M)_t:=\mathrm{exp}
\Big(M_t-\frac{1}{2}\langle M\rangle_t\Big), t\in[0,T],$$
is a uniformly integrable martingale.\\
$\mathrm{2)}$ The energy inequality gives that
$$\EE[\langle M\rangle^n_T]\leqslant n!\Vert M \Vert^{2n}_{BMO}$$
for all $n\in \mathbb{N}^+$, which implies that $BMO\subset \mathbb{H}^p([0,T])$ for every $p\geqslant 1$.\\
$\mathrm{3)}$ According to reverse H\"{o}lder inequality, there exists some $p>1$ such that
$$\EE[\mathcal{E}(M)_T^p]\leqslant C_p,$$
where $C_p$ is a constant only depending on $p$ and the BMO norm of $M$. Moreover, the maximum $p$ satisfying such property can be explicitly
determined by the BMO norm of $M$ through a decreasing function, see more details in $\mathrm{\cite[Theorem 3.1]{Kaza}}$.\\
$\mathrm{4)}$ By John-Nirenberg inequality, we have
$$\EE\Big[\mathcal{E}(M)_T^{-\frac{1}{p-1}}\Big]\leqslant C_p$$
for all $p>1$ satisfying $\Vert M\Vert_{BMO}<\sqrt{2}(\sqrt{p}-1)$, see $\mathrm{\cite[Theorem 2.4]{Kaza}}$.
\end{lemma}

With the powerful tools at hand, we claim the following properties about the solutions to quadratic reflected BSDE (\ref{Y}).

\begin{proposition}\label{Z-BMO}
Suppose Assumptions $\mathrm{(HX)}$ and $\mathrm{(HF)}$ hold, and let $(X,Y,Z,K)$ be the solution of system (\ref{Y}) and (\ref{X}). Then, the stochastic integral
$Z*W:=\Big(\int^t_0 Z_sdW_s\Big)_{t\in[0,T]}$ is a BMO martingale with the BMO norm satisfying
\begin{equation}\label{BMO_norm}
\Vert Z*W\Vert^2_{BMO}\leq \frac{\exp(4\alpha M)}{\alpha^2}   [1+2\alpha M_f(1+M) T].
\end{equation}
\end{proposition}

\begin{proof} Denote $\Vert Y\Vert_\infty \triangleq M$.
Making exponential change of variable $\eta_t:=e^{-\alpha Y_t}$, we
are led to the following reflected BSDE with an upper obstacle
\begin{equation}
\eta_t=\theta_T+\int^T_t F(s,\eta_s,\Lambda_s)ds-\int^T_t \Lambda_s dW_s -(J_T-J_t),\quad t\in[0,T]
\end{equation}
with $\theta_t=e^{-\alpha g(X_t)}, \Lambda_t=-\alpha e^{-\alpha Y_t} Z_t, dJ_t=\alpha e^{-\alpha Y_t} dK_t$ and stochastic coefficient
$$F(t,\omega,y,z)=-\alpha y\Big[f\Big(t,X_t(\omega),\frac{\ln y}{-\alpha},\frac{z}{-\alpha y}\Big)+\frac{\alpha}{2}\Big|\frac{z}{-\alpha
y}\Big|^2\Big] \mathbbm{1}_{\{y \geqslant e^{-\alpha M}\}},$$ which
satisfy $\eta_t \leqslant \theta_t$ and
$$\int^T_0(\theta_t-\eta_t)dJ_t=0.$$ Moreover, by the boundedness of
$Y$ and Assumption $\mathrm{(HF)}$, we have
$$-\alpha M_f(1+M) y-e^{\alpha M}|z|^2 \leqslant
F(t,\omega,y,z)\leqslant\alpha M_f(1+M)  y,$$ and thus
$(\eta,\Lambda, J)\in
\mathbb{S}^\infty[0,T]\times\mathbb{H}^2([0,T];\rm)\times\mathbb{K}^2[0,T]$
with $e^{-\alpha M}\leqslant\eta_t\leqslant e^{\alpha M}$ for all
$t\in[0,T]$.\\

Applying It\^{o}'s formula to $|\eta_t|^2$ gives that
\begin{equation}
|\eta_t|^2=|\eta_T|^2+\int^T_t 2\eta_s F(s,\eta_s,\Lambda_s)ds-\int^T_t 2\eta_s\Lambda_s dW_s-\int^T_t 2\eta_sdJ_s-\int^T_t |\Lambda_s|^2ds.
\end{equation}
Since $dJ_t\geqslant 0,\eta_t >0$ and $\Lambda \in\mathbb{H}^2([0,T];\rm)$, we have
\begin{align*}
\begin{split}
|\eta_t|^2+\Et\Big[\int^T_t |\Lambda_s|^2ds\Big]&
\leqslant\Et|\eta_T|^2+\Et\Big[\int^T_t 2\eta_s F(s,\eta_s,\Lambda_s)ds\Big]\\
&\leqslant \Et|\eta_T|^2+2\alpha M_f(1+M)\Et\Big[\int^T_t |\eta_s|^2ds\Big]\\
&\leqslant [1+2\alpha M_f (1+M)T]\exp(2\alpha M).
\end{split}
\end{align*}
Thus,
\begin{align}
\begin{split}
\Et\Big[\int^T_t |Z_s|^2ds\Big]&=\Et\Big[\int^T_t \Big|\frac{\Lambda_s}{-\alpha \eta_s}\Big|^2ds\Big]\leqslant \frac{\exp(2\alpha
M)}{\alpha^2}  \Et\Big[\int^T_t |\Lambda_s|^2ds\Big]\\
&\leqslant\frac{\exp(4\alpha M)}{\alpha^2}   [1+2\alpha M_f(1+M) T],\quad \forall t\in[0,T],
\end{split}
\end{align}
and one can easily get the conclusion by the definition of BMO martingales.
\end{proof}\\

The above proposition implies that the BMO norm of $Z*W$ depends only on $\alpha, M_f, M_g$, $M$ and $T$. Furthermore, we have the following
$L^p$-integrability of $\int^T_0 |Z_t|^2 dt$ and $K_T$.

\begin{proposition}\label{L^p}
Suppose Assumptions $\mathrm{(HX)}$ and $\mathrm{(HF)}$ hold, and let $(X,Y,Z,K)$ be the solution of system (\ref{Y}) and (\ref{X}). Then, for
any $p\geqslant 1$,
\begin{equation*}
\EE\Big[ \Big(\int^T_0 |Z_t|^2 dt\Big)^p+(K_T)^p\Big]\leqslant C_p.
\end{equation*}
\end{proposition}

\begin{proof}
It is clear from Proposition \ref{Z-BMO} and assertion 2) of Lemma \ref{BMO} to obtain the result of $Z$ part. For the $K$ part, rewrite the
reflected equation as
\begin{align}\label{K}
\begin{split}
K_T&=K_0+Y_0-g(X_T)-\int^T_0 f(s,X_s,Y_s,Z_s)ds+\int^T_0 Z_s dW_s\\
&\leqslant M+M_g+\int^T_0 |f(s,X_s,Y_s,Z_s)|ds+\Big|\int^T_0 Z_s dW_s\Big|\\
&\leqslant M+M_g+M_f T(1+\Vert Y\Vert_\infty)+\frac{\alpha}{2}\int^T_0 |Z_s|^2ds+\Big|\int^T_0 Z_s dW_s\Big|,
\end{split}
\end{align}
where the last line follows from the assumptions of $f$ and $g$. Thus, by Burkholder-Davis-Gundy’s inequality and the conclusion of $Z$ part, we obtain
\begin{align}
\begin{split}
\EE|K_T|^p &\leqslant C_p \Big(1+\EE \Big(\int^T_0 |Z_s|^2 ds\Big)^p+
\EE\Big|\int^T_0 Z_s dW_s\Big|^p\Big)\\
&\leqslant C_p \Big(1+\EE \Big(\int^T_0 |Z_s|^2 ds\Big)^p+\EE \Big(\int^T_0 |Z_s|^2 ds\Big)^{p/2}\Big)\leqslant C_p.
\end{split}
\end{align}
\end{proof}

\section{Main stability result}
Now we are ready to deal with the variation of the solutions to quadratic reflected BSDEs driven by different forward processes. Suppose that
$X^j$ solves
\begin{equation*}
X^j_t=x+\int^t_0 b_j(s,X^j_s)ds+\int^t_0 \sigma_j(s)dW_s,\quad t\in[0,T]
\end{equation*}
for $j=1,2$, where $(b_j,\sigma_j)$ satisfies Assumption (HX), then we know both $X^1$ and $X^2$ are in $\mathbb{S}^2[0,T]$. Given the
parameters $f$ and $g$, let us denote the solutions to the quadratic reflected BSDE (\ref{Y}) driven by $X^1$ and $X^2$ as $(Y^1,Z^1,K^1)$ and
$(Y^2,Z^2,K^2)$ respectively, which belong to $\mathbb{S}^\infty[0,T]\times\mathbb{H}^2([0,T];\rm)\times\mathbb{K}^2[0,T]$ and satisfy $\Vert
Y^1\Vert_\infty\vee\Vert Y^2\Vert_\infty \leqslant M$. We further denote $\delta X= X^1-X^2$, $\delta Y= Y^1-Y^2$, $\delta Z= Z^1-Z^2$ and
$\delta K= K^1-K^2$, and have the following expression
\begin{align}\label{deltaY}
\begin{split}
\delta Y_t
&=\delta Y_T+\int^T_t  f(s,X^1_s,Y^1_s,Z^1_s)-f(s,X^2_s,Y^2_s,Z^2_s)ds-\int^T_t  \delta Z_s dW_s+\int^T_t d\delta K_s\\
&=\delta Y_T+\int^T_t ( \gamma_s \delta X_s+\beta_s\delta Y_s+ \mu_s \delta Z_s )ds-\int^T_t  \delta Z_s dW_s+\int^T_t d\delta K_s,
\end{split}
\end{align}
where
$$\gamma_s :=\frac{f(s,X^1_s,Y^1_s,Z^1_s)-f(s,X^2_s,Y^1_s,Z^1_s)}{X^1_s-X^2_s}\mathbbm{1}_{\{\delta X_s\neq 0\}},$$
$$\beta_s :=\frac{f(s,X^2_s,Y^1_s,Z^1_s)-f(s,X^2_s,Y^2_s,Z^1_s)}{Y^1_s-Y^2_s}\mathbbm{1}_{\{\delta Y_s\neq 0\}},$$
and
$$\mu_s :=\frac{f(s,X^2_s,Y^2_s,Z^1_s)-f(s,X^2_s,Y^2_s,Z^2_s)}{|Z^1_s-Z^2_s|^2}(Z^1_s-Z^2_s)^T\mathbbm{1}_{\{\vert \delta Z_s\vert \neq
0\}}.$$

By the locally Lipschitz assumption of $f$, we have
$$|\gamma_s |\leqslant L(1+|Z^1_s|),\quad |\beta_s |\leqslant L,\quad |\mu_s |\leqslant L(1+|Z^1_s|+|Z^2_s|),\quad \forall s\in[0,T],$$
which further imply that $\int^T_0 (|\gamma_s|^2+|\mu_s|^2) ds$ is $L^p$-integrable for any $p\geqslant 1$ by Proposition \ref{L^p} and that
$\mu*W$ is a BMO martingale by Proposition \ref{Z-BMO}.

Regarding the difficulty caused by the quadratic growth in the $Z$ part, the BMO property of $\mu*W$ enables us to proceed under a new
equivalent probability measure $\mathbb{Q}$ defined as
$$\frac{d\mathbb{Q}}{d\mathbb{P}}:=\mathcal{E}(\mu*W)_T,$$
under which $W^\mathbb{Q}_t=W_t-\int^t_0 \mu_s ds, \ t\in[0,T]$, is a standard Brownian motion.

Moreover, since
$$\Vert\mu*W\Vert_{BMO}\leqslant L(1+\Vert Z^1 *W\Vert_{BMO}+\Vert Z^2*W\Vert_{BMO}),$$
we know from assertion 3) of Lemma \ref{BMO} that there exists some $p^*>1$ which can be determined by the BMO norm of $\mu*W$, such that
$\mathcal{E}(\mu *W)_T$ is $L^{p^*}$-integrable, i.e., $\EE[\mathcal{E}(\mu *W)^{p^*}_T]\leqslant C_{p^*}$. Thus, $C_{p^*}$ only depends on
the BMO norm of $\mu*W$, which essentially relies on the given coefficients $L, \alpha, M_f, M_g$ and $T$, and we may just write it as the
universal constant $C$ hereafter.
Next, we will give the $L^{p}$-estimate of the difference of solutions under such a new probability measure in the following proposition.

\begin{proposition}\label{EQdeltaY}
For any $p>1$ and $\Vert  \delta X \Vert_{\mathbb{S}^{4pq^*}}\leq
1$,
\begin{equation}\label{EQdeltaYZK}
\EQ\Big[\sup_{t\in[0,T]} |\delta Y_t|^{2p}
+\Big(\int^T_0 |\delta Z_t|^2dt\Big)^p
+|\delta K_T|^{2p}\Big]\leqslant C_p\Vert  \delta X \Vert^{p}_{\mathbb{S}^{4pq^*}},
\end{equation}
where $q^*$ is the conjugate exponent of $p^*$, $i.e. \frac{1}{q^*}+\frac{1}{p^*}=1$, and the parameter $p^*$ is determined by the BMO norm of $\mu*W$ (see assertion $\mathrm{3)}$ of Lemma \ref{BMO}).
\end{proposition}

\begin{proof}
To prove this result, we first obtain the estimate of the expectation under the new probability measure
with an undecided parameter $A$ but without taking supremum.
Then, choosing appropriate $A$ gives the exact estimate under supremum norm of the solution component $Y$ in Step two, and the components of the solution $(Z,K)$ in the last
step.
\\

\noindent \textbf{Step one}. Estimates of $\EQ[ |\delta Y_t|^{2p}]$ and $p(2p-1)\EQ[\int^T_0 |\delta Y_t|^{2p-2}|\delta Z_t|^2dt]$.\\

Applying It\^{o}'s formula to $|\delta Y_t|^{2p}$ gives that
\begin{align}\label{deltaY2p}
\begin{split}
|\delta Y_t|^{2p}
=&|\delta Y_T|^{2p}
+2p\int^T_t (\delta Y_s)^{2p-1}\gamma_s\delta X_sds
+2p\int^T_t (\delta Y_s)^{2p-1}\beta_s\delta Y_sds
-2p\int^T_t (\delta Y_s)^{2p-1}\delta Z_sdW^\mathbb{Q}_s\\
&+2p\int^T_t (\delta Y_s)^{2p-1} d\delta K_s-p(2p-1)\int^T_t(\delta Y_s)^{2p-2}|\delta Z_s|^2ds.
\end{split}
\end{align}
Taking expectation under probability measure $\mathbb{Q}$ and recalling the boundedness of $\delta Y$, we have
\begin{align}\label{EdeltaY2p}
\begin{split}
&\EQ[|\delta Y_t|^{2p}]
+p(2p-1)\EQ\Big[\int^T_t |\delta Y_s|^{2p-2}|\delta Z_s|^2ds\Big]\\
\leqslant &\EQ[|\delta Y_T|^{2p}]
+2p\EQ\Big[\int^T_t |\delta Y_s|^{2p-1}|\gamma_s|
|\delta X_s|ds\Big]
+2p\EQ\Big[\int^T_t |\delta Y_s|^{2p}|\beta_s|ds\Big]\\
&+2p\EQ\Big[\int^T_t (\delta Y_s)^{2p-1} d\delta K_s\Big].
\end{split}
\end{align}

Recall the $L^{p^*}$-integrability of $\mathcal{E}(\mu *W)_T$. We have for any $\mathcal{F}_T$-measurable and non-negative
variable $\mathit{X}\in {L}^{q^*}$ that
\begin{equation}\label{q*}
\EQ[\mathit{X}]=\EE[\mathcal{E}(\mu *W)_T\mathit{X}]\leqslant (\EE[\mathcal{E}(\mu *W)^{p^*}_T])^{1/p^*}(\EE[\mathit{X}^{q^*}])^{1/q^*}
\leqslant C (\EE[\mathit{X}^{q^*}])^{1/q^*},
\end{equation}
where $q^*$ is the conjugate exponent of $p^*$.
Moreover, we know that both $X^1$ and $X^2$ are in the space $\mathbb{S}^p [0,T]$ under Assumption (HX) for any $p \geqslant 2$.
Now we are ready to deal with the inequality (\ref{EdeltaY2p}), where the first term can be estimated by the above inequality (\ref{q*}) and
the Lipschitz assumption of $g$ as
\begin{equation}\label{term1}
\EQ[|\delta Y_T|^{2p}]\leqslant L^{2p}\EQ[|\delta X_T|^{2p}]\leqslant C_p(\EE[|\delta X_T|^{2pq^*}])^{1/q^*}
\leqslant C_p
\Vert  \delta X \Vert^{2p}_{\mathbb{S}^{2pq^*}}.
\end{equation}
We also list here the following two estimates for later use. By (\ref{q*}), Cauchy-Schwarz inequality and the $L^p$-integrability of $K^1_T,
K^2_T$ and $\int^T_0|\gamma_t|^2dt$ with arbitrary $p$, we derive that, for any $p>1$
\begin{align}\label{deltaXgamma}
\begin{split}
&\EQ\Big[\sup_{t\in[0,T]}| \delta X_t|^{2p}\Big(\int^T_0  |\gamma_t|^2dt\Big)^p\Big]
\leqslant
C \Big(\EE\Big[\sup_{t\in[0,T]}| \delta X_t|^{2pq^*}\Big(\int^T_0 |\gamma_t|^{2}dt\Big)^{pq^*}\Big]\Big)^{1/q^*}\\
\leqslant &
C \Big(\EE\Big[\sup_{t\in[0,T]}| \delta X_t|^{4pq^*}\Big]\Big)^{1/2q^*}
\Big(\EE\Big[\Big(\int^T_0 |\gamma_t|^{2}dt\Big)^{2pq^*}\Big]\Big)^{1/2q^*}
\leqslant C_p
\Vert  \delta X \Vert^{2p}_{\mathbb{S}^{4pq^*}},
\end{split}
\end{align}
and
\begin{align}\label{deltaXK}
\begin{split}
&\EQ\Big[\sup_{t\in[0,T]}| \delta X_t|^{p}(K^1_T+K^2_T)^p\Big]
\leqslant
C \Big(\EE\Big[\sup_{t\in[0,T]}| \delta X_t|^{pq^*}(K^1_T+K^2_T)^{pq^*}\Big]\Big)^{1/q^*}\\
\leqslant &
C \Big(\EE\Big[\sup_{t\in[0,T]}| \delta X_t|^{2pq^*}\Big]\Big)^{1/2q^*}
\Big(\EE[(K^1_T+K^2_T)^{2pq^*}]\Big)^{1/2q^*}
\leqslant C_p
\Vert  \delta X \Vert^{p}_{\mathbb{S}^{2pq^*}}.
\end{split}
\end{align}

For the second term of (\ref{EdeltaY2p}), we can make use of H\"{o}lder's inequality, Young's inequality and (\ref{deltaXgamma}) to get
\begin{align}\label{term2}
\begin{split}
& 2p\EQ\Big[\int^T_t |\delta Y_s|^{2p-1}|\gamma_s| | \delta X_s|ds\Big]
\leqslant
\EQ\Big[\int^T_t |\delta Y_s|^{2p}ds\Big]
+p^2\EQ\Big[\int^T_t |\delta Y_s|^{2p-2}|\gamma_s|^2| \delta X_s|^2ds\Big]\\
\leqslant &
\EQ\Big[\int^T_t |\delta Y_s|^{2p}ds\Big]
+p^2\EQ\Big[\sup_{t\in[0,T]}|\delta Y_t|^{2p-2}\sup_{t\in[0,T]}| \delta X_t|^2\int^T_0  |\gamma_s|^2ds\Big]\\
\leqslant &
\EQ\Big[\int^T_t |\delta Y_s|^{2p}ds\Big]
+\frac{1}{qA}\EQ\Big[\sup_{t\in[0,T]}|\delta Y_t|^{(2p-2)q}\Big]
+p^{2p-1}A^{p-1}\EQ\Big[\sup_{t\in[0,T]}| \delta X_t|^{2p}\Big(\int^T_0  |\gamma_s|^2ds\Big)^p\Big]\\
\leqslant &
\EQ\Big[\int^T_t |\delta Y_s|^{2p}ds\Big]
+\frac{1}{qA}\EQ\Big[\sup_{t\in[0,T]}|\delta Y_t|^{2p}\Big]
+A^{p-1}C_p\Vert  \delta X \Vert^{2p}_{\mathbb{S}^{4pq^*}},
\end{split}
\end{align}
where $q$ is the conjugate exponent of $p$ and $A>1$ is a constant yet to be determined.

For the third term of (\ref{EdeltaY2p}), we have
\begin{equation}\label{add-Y}
2p\EQ\Big[\int^T_t |\delta Y_s|^{2p}|\beta_s|ds\Big]
\leqslant 2pL \EQ\Big[\int^T_t |\delta Y_s|^{2p}ds\Big].
\end{equation}

Regarding the last term of (\ref{EdeltaY2p}) with reflection, since $g(X^j_s)\leqslant Y^j_s$ for all $s\in[0,T]$ and $K^j$ only increases
when $Y^j=g(X^j)$, $j=1,2$, we can firstly derive that
\begin{equation*}
(Y^1_s-Y^2_s)dK^1_s=[Y^1_s-g(X^1_s)+g(X^1_s)-g(X^2_s)+g(X^2_s)-Y^2_s]dK^1_s\leqslant [g(X^1_s)-g(X^2_s)]dK^1_s,
\end{equation*}
and similarly,
\begin{equation*}
(Y^2_s-Y^1_s)dK^2_s\leqslant [g(X^2_s)-g(X^1_s)]dK^2_s,
\end{equation*}
which imply that
\begin{equation}\label{deltaYdeltaK}
\begin{split}
\delta Y_sd\delta K_s &=(Y^1_s-Y^2_s)dK^1_s+(Y^2_s-Y^1_s)dK^2_s\\
&\leqslant |g(X^1_s)-g(X^2_s)|d(K^1_s+K^2_s)\leqslant L|\delta X_s|d(K^1_s+K^2_s).
\end{split}
\end{equation}
Then, we can use the same arguments as above to obtain
\begin{align}\label{term3}
\begin{split}
& 2p\EQ\Big[\int^T_t (\delta Y_s)^{2p-1} d\delta K_s\Big]
\leqslant
2p L \EQ\Big[\int^T_t |\delta Y_s|^{2p-2}|\delta X_s| d (K^1_s+K^2_s)\Big]\\
\leqslant &
2p L \EQ\Big[\sup_{t\in[0,T]}|\delta Y_t|^{2p-2}
\sup_{t\in[0,T]}| \delta X_t|(K^1_T+K^2_T)\Big]\\
\leqslant &
\frac{1}{qA}\EQ\Big[\sup_{t\in[0,T]}|\delta Y_t|^{(2p-2)q}\Big]
+\frac{(2pL)^p}{p}A^{p-1}
 \EQ\Big[\sup_{t\in[0,T]}| \delta X_t|^p(K^1_T+K^2_T)^p\Big]\\
\leqslant &
\frac{1}{qA}\EQ\Big[\sup_{t\in[0,T]}|\delta Y_t|^{2p}\Big]
+A^{p-1}C_p
\Vert  \delta X \Vert^{p}_{\mathbb{S}^{2pq^*}}.
\end{split}
\end{align}

Plugging (\ref{term1}), (\ref{term2}), (\ref{add-Y}) and (\ref{term3}) back into (\ref{EdeltaY2p}) and by Cauthy-Schwarz inequality, we get
\begin{align*}
\begin{split}
&\EQ[|\delta Y_t|^{2p}]
+p(2p-1)\EQ\Big[\int^T_t |\delta Y_s|^{2p-2}|\delta Z_s|^2ds\Big]\\
\leqslant &
(2pL+1)\int^T_t \EQ|\delta Y_s|^{2p}ds
+\frac{2}{qA}\EQ\Big[\sup_{t\in[0,T]}|\delta Y_t|^{2p}\Big]
+A^{p-1}C_p
\Vert  \delta X \Vert^{p}_{\mathbb{S}^{4pq^*}}, \quad\forall t\in[0,T].
\end{split}
\end{align*}
Since $\delta Y$ is bounded and then $\EQ\Big[\sup_{t\in[0,T]}|\delta Y_t|^{2p}\Big]$ is finite, we obtain from Gronwall's inequality that
\begin{equation}\label{EdeltaY2p-2}
\EQ[|\delta Y_t|^{2p}]\leqslant e^{(2pL+1)T}\Big(\frac{2}{qA}\EQ\Big[\sup_{t\in[0,T]}|\delta Y_t|^{2p}\Big]
+A^{p-1}C_p\Vert  \delta X \Vert^{p}_{\mathbb{S}^{4pq^*}}\Big),\quad \forall t\in[0,T],
\end{equation}
and moreover,
\begin{align}
\begin{split}
&p(2p-1)\EQ\Big[\int^T_0 |\delta Y_s|^{2p-2}|\delta Z_s|^2ds\Big]\\
\leqslant &[(2pL+1)Te^{(2pL+1)T}+1]\Big(\frac{2}{qA}\EQ\Big[\sup_{t\in[0,T]}|\delta Y_t|^{2p}\Big]
+A^{p-1}C_p\Vert  \delta X \Vert^{p}_{\mathbb{S}^{4pq^*}}\Big)\\
\leqslant &
\frac{C_T}{qA}\EQ\Big[\sup_{t\in[0,T]}|\delta Y_t|^{2p}\Big]
+A^{p-1}C_p\Vert  \delta X \Vert^{p}_{\mathbb{S}^{4pq^*}}
\end{split}
\end{align}
with $C_T:= 2[(2pL+1)Te^{(2pL+1)T}+1]$.\\

\noindent \textbf{Step two}. Estimate of $\EQ[\sup_{t\in[0,T]}|\delta Y_t|^{2p}]$.\\

Next, we shall go back to (\ref{deltaY2p}) in the first step of the proof and follow similar procedure to get the estimate of
$\EQ[\sup_{t\in[0,T]}|\delta Y_t|^{2p}]$. To start with, we have
\begin{align}\label{EsupdeltaY2p}
\begin{split}
&\EQ[\sup_{t\in[0,T]}|\delta Y_t|^{2p}]
+p(2p-1)\EQ\Big[\int^T_0 |\delta Y_s|^{2p-2}|\delta Z_s|^2ds\Big]\\
\leqslant &
\EQ[|\delta Y_T|^{2p}]
+2p\EQ\Big[\int^T_0 |\delta Y_s|^{2p-1}|\gamma_s| | \delta X_s|ds\Big]
+2p\EQ\Big[\int^T_0 |\delta Y_s|^{2p}|\beta_s|ds\Big]\\
&+
2p\EQ\Big[\sup_{t\in[0,T]}\Big|\int^T_t (\delta Y_s)^{2p-1}\delta Z_sdW^\mathbb{Q}_s\Big|\Big]
+2p\EQ\Big[\sup_{t\in[0,T]}\int^T_t (\delta Y_s)^{2p-1} d\delta K_s\Big].\\
\end{split}
\end{align}
By (\ref{term2}), (\ref{add-Y}) and (\ref{EdeltaY2p-2}), we have
\begin{align}\label{step2term2}
\begin{split}
&2p\EQ\Big[\int^T_0 |\delta Y_s|^{2p-1}|\gamma_s| | \delta X_s|ds\Big]
+2p\EQ\Big[\int^T_0 |\delta Y_s|^{2p}|\beta_s|ds\Big]\\
\leqslant &
(2pL+1)\EQ\Big[\int^T_0 |\delta Y_s|^{2p}ds\Big]
+\frac{1}{qA}\EQ\Big[\sup_{t\in[0,T]}|\delta Y_t|^{2p}\Big]
+A^{p-1}C_p\Vert  \delta X \Vert^{2p}_{\mathbb{S}^{4pq^*}}\\
\leqslant &
\frac{C_T}{qA}\EQ\Big[\sup_{t\in[0,T]}|\delta Y_t|^{2p}\Big]
+A^{p-1}C_p\Vert  \delta X \Vert^{p}_{\mathbb{S}^{4pq^*}},
\end{split}
\end{align}
similarly by (\ref{deltaYdeltaK}) and (\ref{term3}),
\begin{align}\label{step2term4}
\begin{split}
&2p\EQ\Big[\sup_{t\in[0,T]}\int^T_t (\delta Y_s)^{2p-1} d\delta K_s\Big]
\leqslant
2p\EQ\Big[\sup_{t\in[0,T]}\int^T_t L |\delta Y_s|^{2p-2}|\delta X_s| d (K^1_s+K^2_s)\Big]\\
=&2p L \EQ\Big[\int^T_0 |\delta Y_s|^{2p-2}|\delta X_s| d (K^1_s+K^2_s)\Big]
\leqslant
\frac{1}{qA}\EQ\Big[\sup_{t\in[0,T]}|\delta Y_t|^{2p}\Big]
+A^{p-1}C_p\Vert  \delta X \Vert^{p}_{\mathbb{S}^{2pq^*}}.
\end{split}
\end{align}
As for the stochastic integral term in (\ref{EsupdeltaY2p}), we derive that
\begin{align}\label{step2term3}
\begin{split}
&2p\EQ\Big[\sup_{t\in[0,T]}\Big|\int^T_t (\delta Y_s)^{2p-1}\delta Z_s dW^\mathbb{Q}_s\Big|\Big]
\leqslant  2p \tilde{C}
\EQ\Big[\Big(\int^T_0|\delta Y_s|^{2(2p-1)}|\delta Z_s |^2ds\Big)^{1/2}\Big]\\
\leqslant &
 p \EQ\Big[2\Big(\sup_{t\in[0,T]}|\delta Y_t|^{2p}\int^T_0\tilde{C}^2|\delta Y_s|^{2p-2}|\delta Z_s |^2ds\Big)^{1/2}\Big]\\
\leqslant & p
\Big(\frac{1}{2p-1}
\EQ\Big[\sup_{t\in[0,T]}|\delta Y_t|^{2p}\Big]
+(2p-1)\tilde{C}^2\EQ\Big[\int^T_0|\delta Y_s|^{2p-2}|\delta Z_s |^2ds\Big]\Big),
\end{split}
\end{align}
where $\tilde{C}$ is the constant coming from the following B-D-G inequality
\begin{equation*}
\EQ\Big[\sup_{t\in[0,T]}\Big|\int^T_t \psi_s dW^\mathbb{Q}_s\Big|\Big]
\leqslant \tilde{C}
\EQ\Big[\Big(\int^T_0|\psi_s|^2ds\Big)^{1/2}\Big],
\end{equation*}
which holds for all the $\mathcal{F}_t$-adapted stochastic processes satisfying $\mathbb{Q}\big\{\int^T_0|\psi_s|^2ds<\infty\big\}=1$.
Combining (\ref{EsupdeltaY2p})-(\ref{step2term3}) and together with the results in the first step, we can finally get
\begin{align*}
\begin{split}
\frac{p-1}{2p-1}\EQ[\sup_{t\in[0,T]}|\delta Y_t|^{2p}]
\leqslant &
\frac{C_T +1}{qA}\EQ\Big[\sup_{t\in[0,T]}|\delta Y_t|^{2p}\Big]
+A^{p-1}C_p\Vert  \delta X \Vert^{p}_{\mathbb{S}^{4pq^*}}\\
&+p(2p-1)\tilde{C}^2\EQ\Big[\int^T_0|\delta Y_s|^{2p-2}|\delta Z_s |^2ds\Big]\Big)\\
\leqslant &
[C_T\tilde{C}^2+C_T+1]
\frac{1}{qA}\EQ\Big[\sup_{t\in[0,T]}|\delta Y_t|^{2p}\Big]
+A^{p-1}C_p\Vert  \delta X \Vert^{p}_{\mathbb{S}^{4pq^*}}.
\end{split}
\end{align*}

Choosing $A:=2[C_T\tilde{C}^2+C_T+1]$, we can achieve the desired result for $Y$ part.\\

\noindent \textbf{Step three}. Estimate of $\EQ[(\int^T_0 |\delta Z_s|^2ds)^p]$ and $\EQ[|\delta K_T|^{2p}]$.\\

Applying It\^{o}'s formula to $|\delta Y_t|^2$ gives that
\begin{align*}\label{deltaY^2}
|\delta Y_t|^2
=&|\delta Y_T|^2+\int^T_t 2\delta Y_s(\gamma_s \delta X_s+\beta_s \delta Y_s+\mu_s \delta Z_s)ds-\int^T_t 2\delta Y_s \delta Z_s dW_s\\
&+\int^T_t 2\delta Y_s d\delta K_s-\int^T_t |\delta Z_s|^2ds\\
=&|\delta Y_T|^2+\int^T_t 2\delta Y_s\gamma_s \delta X_sds
+\int^T_t 2|\delta Y_s|^2\beta_s ds
-\int^T_t 2\delta Y_s \delta Z_s dW^\mathbb{Q}_s\\
&+\int^T_t 2\delta Y_s d\delta K_s-\int^T_t |\delta Z_s|^2ds.
\end{align*}
Thus,
\begin{align*}
\begin{split}
\EQ\Big[\Big(\int^T_0 |\delta Z_s|^2ds\Big)^p\Big]
\leqslant
& C_p\Big\{\EQ|\delta Y_T|^{2p}+
\EQ\Big|\int^T_ 0 2\delta Y_s\gamma_s \delta X_sds\Big|^p
+\EQ\Big| \int^T_0 2|\delta Y_s|^2\beta_s ds \Big|^p
\\
&+\EQ\Big|\int^T_ 0 2\delta Y_s\delta Z_sdW^\mathbb{Q}_s\Big|^p
+\EQ\Big|\int^T_0 2\delta Y_s d\delta K_s\Big|^p\Big\}.
\end{split}
\end{align*}
Then, by (\ref{deltaXgamma}), (\ref{deltaXK}) and similar arguments as in the first step, we obtain that
\begin{align*}
\begin{split}
&\EQ\Big|\int^T_ 0 2\delta Y_s\gamma_s \delta X_sds\Big|^p
\leqslant
C_p \EQ\Big[\sup_{t\in[0,T]}| \delta Y_t|^{p}\Big(\int^T_0  |\gamma_s| |\delta X_s|ds\Big)^p\Big]\\
\leqslant & C_p \EQ\Big[\sup_{t\in[0,T]}| \delta Y_t|^{2p}]
+C_p
\EQ\Big[\sup_{t\in[0,T]}| \delta X_t|^{2p}\Big(\int^T_0  |\gamma_s|^2ds\Big)^p\Big]
\leqslant C_p\Vert  \delta X \Vert^{p}_{\mathbb{S}^{4pq^*}},
\end{split}
\end{align*}

\begin{align*}
\begin{split}
\EQ\Big| \int^T_0 2|\delta Y_s|^2\beta_s ds \Big|^p
\leqslant
C_p \EQ\Big[\sup_{t\in[0,T]}| \delta Y_t|^{2p}]
\leqslant  C_p\Vert  \delta X \Vert^{p}_{\mathbb{S}^{4pq^*}},
\end{split}
\end{align*}
and
\begin{align*}
\begin{split}
&\EQ\Big|\int^T_0 2\delta Y_s d\delta K_s\Big|^p
\leqslant
\EQ\Big[\Big( \int^T_0 2L|\delta X_s|d(K^1_s+K^2_s)\Big)^p\Big]
\\
\leqslant
& C_p \EQ\Big[\sup_{t\in[0,T]}| \delta X_t|^{p}(K^1_T+K^2_T)^p\Big]
\leqslant C_p\Vert  \delta X \Vert^{p}_{\mathbb{S}^{2pq^*}}.
\end{split}
\end{align*}
As for the martingale term, by B-D-G and Young's inequality, we  derive that
\begin{align*}
\begin{split}
C_p\EQ\Big|\int^T_ 0 2\delta Y_s\delta Z_sdW^\mathbb{Q}_s\Big|^p
&\leqslant
C_p \EQ\Big[\Big(\int^T_ 0 |\delta Y_s|^2|\delta Z_s|^2ds\Big)^{\frac{p}{2}}\Big]
\leqslant
C_p \EQ\Big[ \sup_{t\in[0,T]}  |\delta Y_t|^p\Big(\int^T_ 0 |\delta Z_s|^2ds\Big)^{\frac{p}{2}}\Big]\\
&\leqslant
C_p \EQ\Big[ \sup_{t\in[0,T]}  |\delta Y_t|^{2p}\Big]
+\frac{1}{2}\EQ\Big[\Big(\int^T_0 |\delta Z_s|^2ds\Big)^p\Big],
\end{split}
\end{align*}
and then together with the result for $Y$ part, we get the conclusion for $Z$.

Regarding the increasing process $K$, notice that we have the expression
\begin{align*}
\begin{split}
\delta K_T
&=
\delta Y_0-[g(X^1_T)-g(X^2_T)]
-\int^T_0[ f(s,X^1_s,Y^1_s,Z^1_s)- f(s,X^2_s,Y^2_s,Z^2_s)]ds
+\int^T_0 \delta Z_s dW_s\\
&=\delta Y_0-[g(X^1_T)-g(X^2_T)]
-\int^T_0 \gamma_s \delta X_sds
-\int^T_0 \beta_s \delta Y_sds
+\int^T_0 \delta Z_s dW^{\mathbb{Q}}_s.
\end{split}
\end{align*}
Thus, by (\ref{term1}), (\ref{deltaXgamma}) and the conclusion for both $Y$ and $Z$ parts, we further obtain
\begin{align*}
\begin{split}
\EQ|\delta K_T|^{2p}
\leqslant & C_p\Big[
\EQ|\delta Y_0|^{2p}
+\EQ|g(X^1_T)-g(X^2_T)|^{2p}
+\EQ\big|\int^T_0 \gamma_s \delta X_sds\big|^{2p}\\
&+\EQ\big|\int^T_0 \beta_s \delta Y_sds\big|^{2p}
+\EQ \big|\int^T_0 \delta Z_s dW^{\mathbb{Q}}_s\big|^{2p}
\Big]\\
\leqslant &
C_p\Big[
\Vert  \delta X \Vert^{p}_{\mathbb{S}^{4pq^*}}
+ \Vert  \delta X \Vert^{2p}_{\mathbb{S}^{2pq^*}}
+\Vert  \delta X \Vert^{2p}_{\mathbb{S}^{4pq^*}}
+\EQ\big[\big(\int^T_0 |\delta Z_s|^2ds\big)^p\big]\Big]
\leqslant C_p
\Vert  \delta X \Vert^{p}_{\mathbb{S}^{4pq^*}},
\end{split}
\end{align*}
which concludes the proof.
\end{proof}\\

Finally, we estimate the variation of the two solutions to quadratic reflected BSDEs constructed with different forward processes, $X^1$ and
$X^2$, under the original probability measure.

\begin{theorem}\label{YandYeZandZe}
Suppose Assumptions $\mathrm{(HX)}$ and $\mathrm{(HF)}$ hold. Then,
for $\Vert  \delta X \Vert_{\mathbb{S}^{4\bar{p}q^*}}\leq 1$, we
have the following conclusion:
\begin{equation*}
\EE\Big[\sup_{t\in[0,T]}|\delta Y_t|^{2}+\int^T_0 |\delta Z_t|^2 dt+|\delta K_T|^{2}\Big]
\leqslant C  \Vert \delta X \Vert_{\mathbb{S}^{4\bar{p}q^*}},
\end{equation*}
where $q^*$ is given in Proposition \ref{EQdeltaY} and $\bar{p}$ is the minimum parameter corresponding to the BMO martingale
$\mu*W$ satisfying assertion $\mathrm{4)}$ of Lemma \ref{BMO}.
\end{theorem}

\begin{proof}
Firstly, let $\bar{p}>1$ be the minimum parameter such that $\Vert\mu * W\Vert_{BMO}<\sqrt{2}(\sqrt{\bar{p}}-1)$ and $\bar{q}$ be its
conjugate exponent.
Note that the constant $C_{\bar{p}}$ appearing in the estimate of the assertion 4) of Lemma \ref{BMO} can be
substituted by a universal constant $C$, as ${\bar{p}}$ can be fully determined by $\Vert\mu * W\Vert_{BMO}$. Then, for any $\mathcal{F}_T$-measurable and non-negative random variable $X\in L^{\bar{p}}$, we obtain that
\begin{align}
\begin{split}
\EE[X]= &\EE[\mathcal{E}(\mu * W)^{1/\bar{p}}_T  X \cdot \mathcal{E}(\mu * W)^{-1/\bar{p}}_T]
\leqslant \Big( \EE\big[\mathcal{E}(\mu * W)_T  X^{\bar{p}}\big] \Big)^{1/\bar{p}}
\Big( \EE\big[\mathcal{E}(\mu * W)^{-\bar{q}/\bar{p}}_T\big] \Big)^{1/\bar{q}}\\
= & \Big( \EQ[X^{\bar{p}}] \Big)^{1/\bar{p}}
\Big( \EE\big[\mathcal{E}(\mu * W)^{-\frac{1}{\bar{p}-1}}_T\big] \Big)^{1/\bar{q}}
\leqslant C \Big( \EQ[X^{\bar{p}}] \Big)^{1/\bar{p}}.
\end{split}
\end{align}
Thus, we can conclude the proof by applying Proposition \ref{EQdeltaY} directly to derive
\begin{align}
\begin{split}
&\EE\Big[\sup_{t\in[0,T]}|\delta  Y_t|^{2}+\int^T_0 |\delta Z_t|^2 dt+|\delta K_T|^{2}\Big]\\
\leqslant  & C
\Big( \EQ\Big[\sup_{t\in[0,T]}|\delta  Y_t|^{2\bar{p}}
+\Big(\int^T_0 |\delta Z_t|^2 dt\Big)^{\bar{p}}
+|\delta K_T|^{2\bar{p}}\Big] \Big)^{1/\bar{p}}
\leqslant C \Vert  \delta X \Vert_{\mathbb{S}^{4\bar{p}q^*}}.
\end{split}
\end{align}

\end{proof}

\section{Application to numerical scheme for quadratic reflected BSDEs}

In this section, we apply the quantitative stability result to the convergence analysis for a discrete-time numerical
scheme for the quadratic reflected BSDE (\ref{Y})-(\ref{X}) under Markovian framework and Assumptions (HX) and (HF). For the sake of further time
discretization, we need to assume in this section that $b, \sigma$ and $f$ satisfy H\"{o}lder's continuity with respect to the time variable. That
is, for any $0\leqslant s \leqslant t\leqslant T$ and any $(x,y,z) \in \mathbb{R}\times \mathbb{R}\times \rm$,
\begin{equation}\tag{HT}
|b(t,x)-b(s,x)|+|\sigma(t)-\sigma(s)|+|f(t,x,y,z)-f(s,x,y,z)|
\leqslant L(t-s)^{\frac{1}{2}}.
\end{equation}

Different from quadratic BSDEs without reflection and Lipschitz BSDEs with reflection, where the solution component $Z$ is typically bounded in Markovian setup, the solution component $Z$ for quadratic reflected BSDE (\ref{Y})-(\ref{X}) is not necessarily bounded. This is the major difficulty to propose a numerical scheme and study its convergence. To overcome this difficulty, we resort to the discretely reflected version of BSDE introduced in (\ref{Y^R}), where the reflection is only permitted to operate at specific discrete time points.
In \cite{DQ}, we have proved that the corresponding solution $(Y^\mathcal{R},Z^\mathcal{R})$ is actually a good approximation of its
continuous counterpart $(Y,Z)$ in (\ref{Y}) (Note that the generator $f$ does not involve $y$ in \cite{DQ}, but one can easily extend the
result therein to include $y$ in the generator). Moreover, since the solution component $Z^\mathcal{R}$ is uniformly bounded, we can truncate the corresponding generator via the bound of $Z^{\mathcal{R}}$ and obtain a truncated discrete-time numerical scheme on each reflected interval. The quantitative stability result will play a pivotal role for the convergence analysis of this numerical scheme.
Firstly, let us give some basic definitions which will be used later.

\subsection{Definition and notations}
Given a partition $\pi:=\{0=t_0<t_1<\cdots<t_N=T\}$ of $[0, T]$, we shall first introduce the standard Euler scheme $X^\pi$ for $X$, which has
been widely studied in the literature and has the form
\begin{equation*}
\left\{
\begin{array}{lr}
X^\pi_0=x, & \\
X^\pi_{t_{i+1}}=X^\pi_{t_{i}}+b(t_i,X^\pi_{t_{i}})(t_{i+1}-t_{i})+\sigma(t_i)(W_{t_{i+1}}-W_{t_{i}}),\quad i\leqslant N-1,&
\end{array}
\right.
\end{equation*}
whose continuous-time version is defined correspondingly as
\begin{equation*}\label{X^pi}
X^\pi_t=X^\pi_{t_{i}}+b(t_i,X^\pi_{t_{i}})(t-t_{i})+\sigma(t_i)(W_t-W_{t_{i}}),\quad t\in[t_i, t_{i+1}),\ i\leqslant N-1.
\end{equation*}
Denote $|\pi|:=\max_{i\leqslant N-1} (t_{i+1}-t_{i})$ and without loss of generality, assume that $N|\pi|\leqslant L$.  Then under Assumption
(HX), we know that $X^\pi\in \mathbb{S}^{2p} [0,T]$ and
\begin{equation}\label{Eulerest}
\EE\Big[\sup_{0\leqslant t\leqslant T} |X_t-X^\pi_t|^{2p}\Big]+\max_{0\leqslant i\leqslant N-1}\EE\Big[\sup_{t\in[t_i,t_{i+1}]}
|X_t-X^\pi_{t_i}|^{2p}\Big]\leqslant C|\pi|^p,\quad p\geqslant 1,
\end{equation}
see e.g., \cite{KP}. Furthermore,  with piecewisely constant coefficients, we may regard (\ref{X^pi}) as a special case of (\ref{X}) with
coefficients satisfying (HX).\\

Next, we define $(Y^e,Z^e,K^e)$ as the solution to the following continuously reflected BSDE driven by $X^\pi$, instead of $X$ in (\ref{Y}),
\begin{align}\label{Y^e}
\begin{split}
&Y^e_t=g(X^\pi_T)+\int^T_t f(s,X^\pi_s,Y^e_s,Z^e_s)ds-\int^T_t Z^e_s dW_s +K^e_T-K^e_t, \\
&Y^e_t\geqslant g(X^\pi_t),\quad \int^T_0 (Y^e_t-g(X^\pi_t))dK^e_t=0.
\end{split}
\end{align}
Since $X^\pi \in \mathbb{S}^2[0,T]$ and the system (\ref{Y})-(\ref{X}) is decoupled, we know that $(Y^e,Z^e,K^e)\in
\mathbb{S}^\infty[0,T]\times\mathbb{H}^2([0,T];\rm)\times\mathbb{K}^2[0,T]$ and can further obtain \emph{a priori} estimates from Propositions
\ref{Z-BMO} and \ref{L^p}, i.e.,
$$\Vert Y^e\Vert_\infty\leqslant M, \quad\Vert Z^e*W\Vert^2_{BMO}\leqslant \frac{\exp(4\alpha M)}{\alpha^2}   [1+2\alpha M_f (1+M)T]$$
and the $L^p$-integrability of $K^e_T$ and $\int^T_0 |Z^e_t|^2 dt$ for any $p\geqslant 1$. \\

Now, we are in a position to introduce the aforementioned discretely reflected BSDE, which is defined recursively and only operates at
specific times $\mathcal{R}=\{r_j,0\leqslant j\leqslant \kappa\ \vert\ 0=r_0<r_1\cdots<r_{\kappa-1}<r_\kappa=T\}$. Let
$|\mathcal{R}|:=\max_{j\leqslant \kappa -1}(r_{j+1}-r_j)$ and for the sake of further discussion, suppose that $\mathcal{R} \subset \pi$,
which means the discrete reflection times are all included in the partition time points. The solution $(Y^{\mathcal{R}},Z^{\mathcal{R}})$
satisfies
$$Y^{\mathcal{R}}_T=\tilde{Y}^{\mathcal{R}}_T=g(X_T),$$
and for $j\leqslant \kappa-1$, $t\in[r_j,r_{j+1})$,
\begin{equation}\label{Y^R}
\left\{
\begin{array}{lr}
\tilde{Y}^{\mathcal{R}}_t=Y^{\mathcal{R}}_{r_{j+1}}
+\int^{r_{j+1}}_t f(s,X_s,\tilde{Y}^{\mathcal{R}}_s,Z^{\mathcal{R}}_s)ds
-\int^{r_{j+1}}_t Z^{\mathcal{R}}_s dW_s,
&  \\
Y^{\mathcal{R}}_t=\tilde{Y}^{\mathcal{R}}_t
+[g(X_t)-\tilde{Y}^{\mathcal{R}}_t]^ + \mathbbm{1}_{\{t\in{\mathcal{R}}\}}. &
\end{array}
\right.
\end{equation}

For later use, we also define the
solution $(Y^{\mathcal{R},e},Z^{\mathcal{R},e})$ to discretely reflected BSDE, which is the same as defined in (\ref{Y^R}) but with $X$
substituted by $X^\pi$, i.e.,
$$Y^{\mathcal{R},e}_T=\tilde{Y}^{\mathcal{R},e}_T=g(X^{\pi}_T),$$
and for $j\leqslant \kappa-1$, $t\in[r_j,r_{j+1})$,
\begin{equation}\label{Y^Re}
\left\{
\begin{array}{lr}
\tilde{Y}^{\mathcal{R},e}_t=Y^{\mathcal{R},e}_{r_{j+1}}
+\int^{r_{j+1}}_t f(s,X^{\pi}_s,\tilde{Y}^{\mathcal{R},e}_s,Z^{\mathcal{R},e}_s)ds
-\int^{r_{j+1}}_t Z^{\mathcal{R},e}_s dW_s,
&  \\
Y^{\mathcal{R},e}_t=\tilde{Y}^{\mathcal{R},e}_t
+[g(X^{\pi}_t)-\tilde{Y}^{\mathcal{R},e}_t]^+ \mathbbm{1}_{\{t\in{\mathcal{R}}\}}.&
\end{array}
\right.
\end{equation}
To simplify the expression, let us denote the discretely reflected BSDE systems (\ref{Y^R}) and (\ref{Y^Re}) as $DR(f, g, X)$ and $DR(f, g,
X^\pi)$, respectively. \\

Next, we apply truncation technique to handle the locally Lipschitz and quadratic growth condition. Define $f_n(t,x,y,z):=f(t,x,y,h_n(z))$ for
all $(t,x,y,z)\in[0,T]\times \mathbb{R}\times\mathbb{R}\times\rm$, where $h_n$ is a smooth modification of the projection on the centered ball
of radius $n$ such that $\vert h_n(z)\vert\leqslant n+1$, $\vert\nabla h_n\vert\leqslant 1$ and satisfying that $h_n(z)=z $ when $|z|\leqslant
n$, for all $n\in \mathbb{R}^+$. Thus, we can define analogously the truncated discretely reflected BSDEs $DR(f_n, g, X)$ and $DR(f_n,g,
X^\pi)$ in order to meet the Lipschitz condition, and denote their solutions by
$(Y^{\mathcal{R},n},Z^{\mathcal{R},n})$ and $(Y^{\mathcal{R},e,n},Z^{\mathcal{R},e,n})$, respectively.\\

Furthermore, thanks to \cite[Lemma 4.5]{DQ}, we know that the second component $Z^\mathcal{R}$
of the solution to discretely reflected BSDE is uniformly bounded with regard to the discrete reflection $\mathcal{R}$ and
the bound $M_z$ only depends on the given coefficients in our Assumptions (HF) and (HX). One can easily check that this result also holds true
for $Z^{\mathcal{R},e}$. Thus, taking $n=M_z$, we know immediately that $(Y^{\mathcal{R},M_z},Z^{\mathcal{R},M_z})$ (resp.
$(Y^{\mathcal{R},e,M_z},Z^{\mathcal{R},e,M_z})$) coincides with $(Y^{\mathcal{R}},Z^{\mathcal{R}})$ (resp.
$(Y^{\mathcal{R},e},Z^{\mathcal{R},e})$) and therefore we only need to focus on the discrete-time scheme for such truncated discretely
reflected BSDE with parameter $M_z$ and generator $f_{M_z}$, which satisfies, for all $(x,y,z),\
(x^\prime,y^\prime,z^\prime)\in\mathbb{R}\times\mathbb{R}\times\rm$, that
\begin{equation*}
|f_{M_z}(t,x,y,z)-f_{M_z}(t,x^\prime,y^\prime,z^\prime)|\leqslant L(M_z+2) |x-x^\prime|+L |y-y^\prime|+L(2M_z+3) |z-z^\prime|.
\end{equation*}

\subsection{Truncated discrete-time numerical scheme}
Inspired by classical numerical schemes under Lipschitz condition (See \cite{BT} \cite{Gobet} and \cite{Zhang} for BSDEs and \cite{BBJC}
\cite{Ma} for reflected BSDEs) and the truncated discretely reflected BSDE in the above subsection, we now introduce the following truly
discretized scheme with the help of truncation function $h_{M_z}$. We define a pair of piecewise constant process $(\bar{Y}^\pi, \bar{Z}^\pi)$ recursively via
$$\bar{Y}^\pi_{t_{N}}=\tilde{Y}^\pi_{t_{N}}=g(X^\pi_T)$$
and
\begin{equation}\label{discheme}
\left\{
\begin{array}{lr}
\bar{Z}^\pi_{t_{i}}=(t_{i+1}-t_{i})^{-1}\EE_{t_{i}}\big[\bar{Y}^\pi_{t_{i+1}}(W_{t_{i+1}}-W_{t_{i}})\big], &\\
\tilde{Y}^\pi_{t_{i}}=\EE_{t_{i}}\big[\bar{Y}^\pi_{t_{i+1}}\big]+(t_{i+1}-t_{i})f\big(t_i,X^\pi_{t_i},\tilde{Y}^\pi_{t_{i}},h_{M_z}(\bar{Z}^\pi_{t_{i}})\big),
&\quad i \leqslant N-1,\\
\bar{Y}^\pi_{t_{i}}=\tilde{Y}^\pi_{t_{i}}+\big[g(X^\pi_{t_i})-\tilde{Y}^\pi_{t_{i}}\big]^+ \mathbbm{1}_{\{t_i \in{\mathcal{R}}\}}, &
\end{array}
\right.
\end{equation}
and setting
$$(\bar{Y}^\pi_t, \bar{Z}^\pi_t)=(\bar{Y}^\pi_{t_i}, \bar{Z}^\pi_{t_i})\quad  \text{for} \ t\in[t_i,t_{i+1}),\ i\leqslant N-1.$$

For later use, we shall introduce the continuous-time scheme associated with the square integrable processes $(\bar{Y}^\pi, \bar{Z}^\pi)$. By
the martingale representation theorem, we know that there exists $Z^\pi\in\mathbb{H}^2([t_i,t_{i+1});\rm)$ such that
$$\bar{Y}^\pi_{t_{i+1}}=\EE_{t_{i}}\big[\bar{Y}^\pi_{t_{i+1}}\big]+\int^{t_{i+1}}_{t_i} Z^\pi_u dW_u, \quad i\leqslant N-1.$$
We can then define $\tilde{Y}^\pi$ and $Y^\pi$ for $[t_i,t_{i+1}),\ i\leqslant N-1$ by

\begin{equation}\label{ctscheme}
\left\{
\begin{array}{lr}
\tilde{Y}^\pi_{t}=\bar{Y}^\pi_{t_{i+1}}+(t_{i+1}-t)f_{M_z}\big(t_i,X^\pi_{t_i},\tilde{Y}^\pi_{t_{i}},\bar{Z}^\pi_{t_{i}}\big)-\int^{t_{i+1}}_{t}
Z^\pi_u dW_u,
&\\
Y^\pi_{t}=\tilde{Y}^\pi_{t}+\big[g(X^\pi_{t})-\tilde{Y}^\pi_{t}\big]^+ \mathbbm{1}_{\{t\in{\mathcal{R}}\}}. &
\end{array}
\right.
\end{equation}
One can check the connection between (\ref{discheme}) and (\ref{ctscheme}): $Y^\pi=\bar{Y}^\pi$ on $\pi$ and $Y^\pi=\tilde{Y}^\pi$ on
$[0,T]\setminus \mathcal{R}$, and by It\^o's isometry,
\begin{equation*}\label{barZpi}
\bar{Z}^\pi_t=\bar{Z}^\pi_{t_i}=(t_{i+1}-t_{i})^{-1} \EE_{t_i}\Big[\int^{t_{i+1}}_{t_i} Z^\pi_u du\Big],\quad t\in[t_i,t_{i+1}),\ i\leqslant
N-1.
\end{equation*}
Moreover, we define the piecewise constant process for $Z^\mathcal{R}$ likewise by
\begin{equation*}\label{barZR}
\bar{Z}^\mathcal{R}_t:=(t_{i+1}-t_{i})^{-1} \EE_{t_i}\Big[\int^{t_{i+1}}_{t_i} Z^\mathcal{R}_u du\Big],\quad  t\in[t_i,t_{i+1}),\ i\leqslant
N-1.
\end{equation*}
which is known as the best $\mathbb{H}^2$-approximation of $Z^\mathcal{R}$.

\subsection{Approximation results for discretely reflected BSDEs}
It has been shown in \cite{DQ} that discretely reflected BSDE is actually a good approximation of continuously reflected
BSDE. Thus, we shall first consider the convergence from the numerical scheme (\ref{discheme}) to the discretely reflected BSDE in this
subsection. With the aid of the boundedness of $Z^\mathcal{R}$ and its truncation, we can indeed proceed under Lipschitz condition now.

There are already results about the convergence for discretely reflected BSDEs driven by $X$ and $X^\pi$ under Lipschitz condition, see
\cite[Theorem 3.1 and Corollary 3.1]{BBJC}, where the authors first show that the approximation error for the discretely reflected BSDE constructed with
$X$ (resp. $X^\pi$) is ultimately controlled by $\Vert Z^\mathcal{R}-\bar{Z}^\mathcal{R}\Vert_{\mathbb{H}^2}$ (resp. $\Vert
Z^{\mathcal{R},e}-\bar{Z}^{\mathcal{R},e}\Vert_{\mathbb{H}^2}$), and then by means of the representation for $Z^\mathcal{R}$(resp.
$Z^{\mathcal{R},e}$) in terms of the next reflection time to obtain the regularity result. We may now directly apply the result to our
truncated dicrete-time scheme under the Assumptions (HX), (HF), (HT) and the following additional assumptions.

\begin{assumption}
$g$ and $\sigma$ further satisfy: \\
\noindent$\mathrm{(H1)}$ $g\in C^1_b$ with $L$-Lipschitz derivative.\\
\noindent$\mathrm{(H2)}$ $g\in C^2_b$ with $L$-Lipschitz first and second derivatives, $\sigma$ satisfies $L$-Lipschitz condition with respect
to time variable.\\
\end{assumption}

\begin{lemma}\label{Quaresult}
Suppose $\mathrm{(HX)}$, $\mathrm{(HF)}$ and $\mathrm{(HT)}$ hold. Then,
\begin{align*}
\begin{split}
&\max_{ j \leqslant \kappa-1}\Vert\sup_{t\in[r_j,r_{j+1}]} |Y^\mathcal{R}_t-Y^{\pi}_t|\Vert_{\mathbb{L}^2}
+\max_{ i \leqslant N-1}\Vert\sup_{t\in(t_i,t_{i+1}]} |Y^\mathcal{R}_t-\bar{Y}^{\pi}_{t_{i+1}}|\Vert_{\mathbb{L}^2}
\leqslant C\big(\alpha_1(\kappa)\vert \pi\vert^{\frac{1}{2}}+\epsilon_1(\pi)\big),\\
&\Vert Z^\mathcal{R}-Z^{\pi}\Vert_{\mathbb{H}^2}+\Vert Z^\mathcal{R}-\bar{Z}^{\pi}\Vert_{\mathbb{H}^2}
\leqslant C\big(\alpha_2(\kappa)\vert \pi\vert^{\frac{1}{2}}
+\epsilon_1(\pi)\big),\\
&\Vert Z^{\mathcal{R},e}-Z^{\pi}\Vert_{\mathbb{H}^2}+\Vert Z^{\mathcal{R},e}-\bar{Z}^{\pi}\Vert_{\mathbb{H}^2}
\leqslant C \big(\alpha_1(\kappa)\vert \pi\vert^{\frac{1}{2}}+\epsilon_2(\pi)\big),
\end{split}
\end{align*}
with $(\alpha_1(\kappa),\alpha_2(\kappa), \epsilon_1(\pi),\epsilon_2(\pi))=(\kappa^{\frac{1}{4}},\kappa^{\frac{1}{2}},\vert
\pi\vert^{\frac{1}{4}},\vert \pi\vert^{\frac{1}{4}})$
under $\mathrm{(H1)}$, and $(\alpha_1(\kappa),\alpha_2(\kappa),
\epsilon_1(\pi),\epsilon_2(\pi))=(1,\kappa^{\frac{1}{2}},\vert\pi\vert^{\frac{1}{2}},\vert \pi\vert^{\frac{1}{4}})$
under $\mathrm{(H2)}$.
\end{lemma}

\begin{proof}
Keeping $(Y^{\mathcal{R}},Z^{\mathcal{R}})=(Y^{\mathcal{R},M_z},Z^{\mathcal{R},M_z})$ and $Z^{\mathcal{R},e}=Z^{\mathcal{R},e,M_z}$ in mind
and applying the main theorem in \cite{BBJC} under Lipschitz case to our truncated scheme (\ref{discheme}) and (truncated) discretely
reflected BSDE (\ref{Y^R})/(\ref{Y^Re}),
we can obtain the conclusion directly.
\end{proof}

\subsection{Approximation results for continuously reflected BSDEs}

We first recall our previous result about the convergence rate from discretely to continuously reflected
BSDE in \cite{DQ}. As mentioned before, one can readily verify that all the results therein still hold  when we replace $X$ by $X^\pi$ and
under the general driver $f$ involving $y$.

\begin{lemma}\label{YandYR}
Let $\mathrm{(HX)}$ and $\mathrm{(HF)}$ hold. Then,
\begin{align*}
\begin{split}
&\max_{ j \leqslant \kappa-1} \Vert\sup_{t\in[r_j,r_{j+1}]} |Y_t-Y^\mathcal{R}_t|\Vert_{\mathbb{L}^2}
+
\Vert Z-Z^{\mathcal{R}}\Vert_{\mathbb{H}^2}
\leqslant C|\mathcal{R}|^{\frac{1}{4}},\\
&\max_{ j \leqslant \kappa-1} \Vert\sup_{t\in[r_j,r_{j+1}]} |Y^e_t-Y^{\mathcal{R},e}_t|\Vert_{\mathbb{L}^2}
+
\Vert Z^e-Z^{\mathcal{R},e}\Vert_{\mathbb{H}^2}
\leqslant C|\mathcal{R}|^{\frac{1}{4}}.\\
\end{split}
\end{align*}
In addition, if Assumption $\mathrm{(H1)}$ holds, the index of convergence rate will become $\frac{1}{2}$.
\end{lemma}

Note that the conclusion under Assumption (H1) can be obtained from \cite[Theorem 4.6]{DQ} by using an approximation argument as usual.
Finally, we present our main theorem of this section regarding the convergence result of the numerical scheme to continuously reflected BSDE
with quadratic growth and deterministic $\sigma$. In order to keep consistency between the two different convergence criteria appearing
in the above lemmas \ref{Quaresult} and \ref{YandYR}, we assume the reflection points and the partition points coincide, i.e. $\mathcal{R}=\pi$ (thus $\kappa=N$) in the
following theorem.

\begin{theorem}
Suppose $\mathrm{(HX)}$, $\mathrm{(HF)}$, $\mathrm{(HT)}$ and $\mathrm{(H1)}$ hold. Then, the following estimates hold with $q=\frac{1}{4}$:
\begin{align*}
\begin{split}
&\max_{ i \leqslant N-1} \Vert\sup_{t\in[t_i,t_{i+1}]} |Y_t-Y^{\pi}_t|+\sup_{t\in(t_i,t_{i+1}]}
|Y_t-\bar{Y}^{\pi}_{t_{i+1}}|\Vert_{\mathbb{L}^2}
\leqslant C|\pi|^q,\\
&\Vert Z-Z^{\pi}\Vert_{\mathbb{H}^2}+\Vert Z-\bar{Z}^{\pi}\Vert_{\mathbb{H}^2}
\leqslant C|\pi|^{\frac{1}{4}}.
\end{split}
\end{align*}
Moreover, under Assumption $\mathrm{(H2)}$, we have finer result for $Y$ part with $q=\frac{1}{2}$.
\end{theorem}
\begin{proof}\\
\noindent $Y$ part:
Lemma \ref{Quaresult} and \ref{YandYR} lead straightforward to the result for $Y$ with $\mathcal{R}=\pi$.\\

\noindent $Z$ part: As shown in Lemma \ref{Quaresult}, one cannot get the final convergence with only $Z^\mathcal{R}$ due to
the troublesome term $\kappa^{\frac{1}{2}} |\pi|^{\frac{1}{2}}$ on the righthand side of the estimate. Note that the problem cannot be resolved by simply increasing
the regularity assumption on $g$ and $\sigma$. Thus, we need to proceed with the help of $Z^{\mathcal{R},e}$.

Taking $(X^1,Y^1,Z^1)=(X,Y,Z)$, $(X^2,Y^2,Z^2)=(X^\pi,Y^e,Z^e)$ and applying the stability result in Theorem \ref{YandYeZandZe}, we obtain
from the estimate (\ref{Eulerest}) that
$$\Vert Z-Z^e\Vert^2_{\mathbb{H}^2}\leqslant C\Vert  X-X^\pi \Vert_{\mathbb{S}^{4\bar{p}q^*}}
\leqslant C |\pi|^{\frac{1}{2}},$$
where $\bar{p}$ and $q^*$ are given in Theorem 3.2. Then, the conclusion for the solution component $Z$ follows from the results related to $Z^{\mathcal{R},e}$
in Lemma \ref{Quaresult} and \ref{YandYR}. This completes the proof.
\end{proof}

\section{Conclusions}

In this paper, we propose a truncated
discrete-time numerical scheme for quadratic reflected BSDEs. To prove the convergence, we develop a quantitative stability result for the quadratic reflected BSDE, and then adapt the numerical analysis for quadratic BSDEs without reflection and Lipschitz BSDEs with reflection. One of the critical conditions is the deterministic assumption on the volatility term $\sigma$, which is imposed to guarantee
the uniform
boundedness for the solution component $Z^{\mathcal{R}}$ in the corresponding discretely reflected BSDE. A natural extension is to consider the multiplicative $\sigma$ by allowing it to depend on the underlying states. This is far more
challenging and the major difficulty is to obtain a uniform estimate for $Z^{\mathcal{R}}$ with regard to the discrete reflection $\mathcal{R}$. Such an extension is left for the future research.


%
%
%
%
%
%
%


\small
\begin{thebibliography}{1}

\bibitem{ErhanSong}
E. Bayraktar and S. Yao. {Quadratic reflected BSDEs with unbounded obstacles}. {\it Stochastic Processes and their Applications},
$\mathbf{122}$(4): 1155-1203, 2012.

\bibitem{BBJC}
B. Bouchard and J. F. Chassagneux. {Discrete-time approximation for continuously and discretely reflected BSDEs}. {\it Stochastic Processes
and their Applications},  $\mathbf{118}$: 2269-2293, 2008.

\bibitem{BT}
B. Bouchard and N. Touzi. {Discrete-time approximation and Monte-Carlo simulation of backward stochastic differential equations}. {\it
Stochastic Processes and their Applications}, $\mathbf{111}$(2): 175-206, 2004.

\bibitem{PB1}
P. Briand and Y. Hu. {BSDE with quadratic growth and unbounded terminal value}. {\it Probability Theory and Related Fields},
$\mathbf{136}$(4): 604-618, 2006.

\bibitem{PB}
P. Briand and Y. Hu. {Quadratic BSDEs with convex generators and unbounded terminal conditions}. {\it Probability Theory and Related Fields},
$\mathbf{141}$: 543-567, 2008.

\bibitem{JCAR}
J. F. Chassagneux and A. Richou. {Numerical simulation of quadratic BSDEs}, {\it The Annals of Applied Probability}, $\mathbf{26}$(1):
262-304, 2016.

\bibitem{Gobet}
E. Gobet, J. P. Lemor and X. Warin. {A regression-based Monte Carlo method to solve backward stochastic differential equations}, {\it The
Annals of Applied Probability}, $\mathbf{15}$(3): 2172-2202, 2005.

\bibitem{HLW}
Y. Hu, X. Li and J. Wen. Anticipated backward stochastic differential equations with quadratic growth. {\it Journal of Differential Equations}, $\mathbf{270}$: 1298-1331, 2021.

\bibitem{IDR}
P. Imkeller and G. Dos Reis. {Path regularity and explicit convergence rate for BSDE with truncated quadratic growth}. {\it Stochastic
Processes and their Applications}, $\mathbf{120}$(3): 348-379, 2010.

\bibitem{El}
N. El Karoui, C. Kapoudjian, E. Pardoux, S. Peng and M. C. Quenez. {Reflected solutions of backward SDE's, and related obstacle problems for
PDE's}. {\it The Annals of Probability}, $\mathbf{25}$(2): 702-737, 1997.

\bibitem{LX}
J. P. Lepeltier and M. Xu. {Reflected BSDE with quadratic growth and unbounded terminal value}, { arXiv:0711.0619v1}, 2007.

\bibitem{Kaza}
N. Kazamaki. {Continuous exponential martingales and BMO}. {\it Lecture Notes in Mathematics}, $\mathbf{1579}$, 1994.

\bibitem{KP}
P. E. Kloeden and E. Platen. {Numerical Solution of Stochastic Differential Equations}. {\it Applications of Mathematics (New York)}
$\mathbf{23}$. Springer, Berlin.

\bibitem{KobyBSDE}
M. Kobylanski. {Backward stochastic differential equations and partial differential equations with quadratic growth}. {\it The Annals of
Probability}, $\mathbf{28}$(2): 558-602, 2000.

\bibitem{Koby}
M. Kobylanski, J. P. Lepeltier, M. C. Quenez and S. Torres. {Reflected BSDE with superlinear quadratic coefficient}. {\it Probability and
Mathematical Statistics}, $\mathbf{22}$(1): 51-83, 2002.

\bibitem{Ma}
J. Ma and J. Zhang. {Representations and regularities for solutions to BSDEs with reflections}, {\it Stochastic Processes and their
Applications}, $\mathbf{115}$(4): 539-569, 2005.

\bibitem{Richou2011}
A. Richou. {Numerical simulation of BSDEs with drivers of quadratic growth}. {\it The Annals of Applied Probability}, $\mathbf{21}$(5):
1933-1964, 2011.

\bibitem{DQ}
D. Sun. {The convergence rate from discrete to continuous optimal investment stopping problem}, {\it Chinese Annals of Mathematics, Series B}, $\mathbf{42}$, 259–280, 2021.

\bibitem{Zhang}
J. Zhang. A numerical scheme for BSDEs. {\it The Annals of Applied Probability}, $\mathbf{14}$(1): 459-488, 2004.

\end{thebibliography}
\end{document}